\documentclass[12pt]{amsart}

\usepackage{cite}
\usepackage[hypertexnames=false, colorlinks, linkcolor=blue, allcolors=blue]{hyperref}
\hypersetup{nesting=true,debug=true,naturalnames=true}
\usepackage[margin=1in]{geometry}
\usepackage{graphicx}
\usepackage{enumitem}
\usepackage{thm-restate}
\usepackage{cleveref}

\declaretheorem[name=Theorem]{thm}

\newtheorem{lemma}{Lemma}

\newtheorem{definition}{Definition}
\newtheorem{proposition}{Proposition}
\newtheorem{corollary}{Corollary}
\newtheorem{question}{Question}

\theoremstyle{definition}
\newtheorem{example}{Example}

\newtheorem{remark}{Remark}

\begin{document}

\title{Equivariant cobordisms between freely-periodic knots} 

\author{Keegan Boyle and Jeffrey Musyt}  

\address[Keegan Boyle]{Department of Mathematics, University of British Columbia, Canada} 
\email{kboyle@math.ubc.ca}
\address[Jeffrey Musyt]{Department of Mathematics and Statistics, Slippery Rock University, USA} 
\email{jeffrey.musyt@sru.edu}

\setcounter{section}{0}

\begin{abstract}
We consider free symmetries on cobordisms between knots. We classify which freely periodic knots bound equivariant surfaces in the 4-ball in terms of corresponding homology classes in lens spaces. A key tool is the homology cobordism classification of lens spaces using d-invariants. We give a numerical condition determining the free periods for which torus knots bound equivariant surfaces in the 4-ball. 
\end{abstract}
\maketitle

\section{Introduction}

A knot $K \subset S^3$ is \emph{freely $(p,q)$-periodic} if there is a free $\mathbb{Z}/p\mathbb{Z}$-action on $S^3$ with quotient $L(p,q)$ which leaves $K$ invariant. An example is shown in Figure \ref{fig:T(2,7)}. Thinking of $S^3$ as the unit sphere in $\mathbb{C}^2$, this symmetry is conjugate to $(z,w) \mapsto (\alpha z, \alpha^q w)$, where $\alpha = e^{2 \pi i/ p}$. 
\begin{figure} 
\scalebox{.5}{\includegraphics{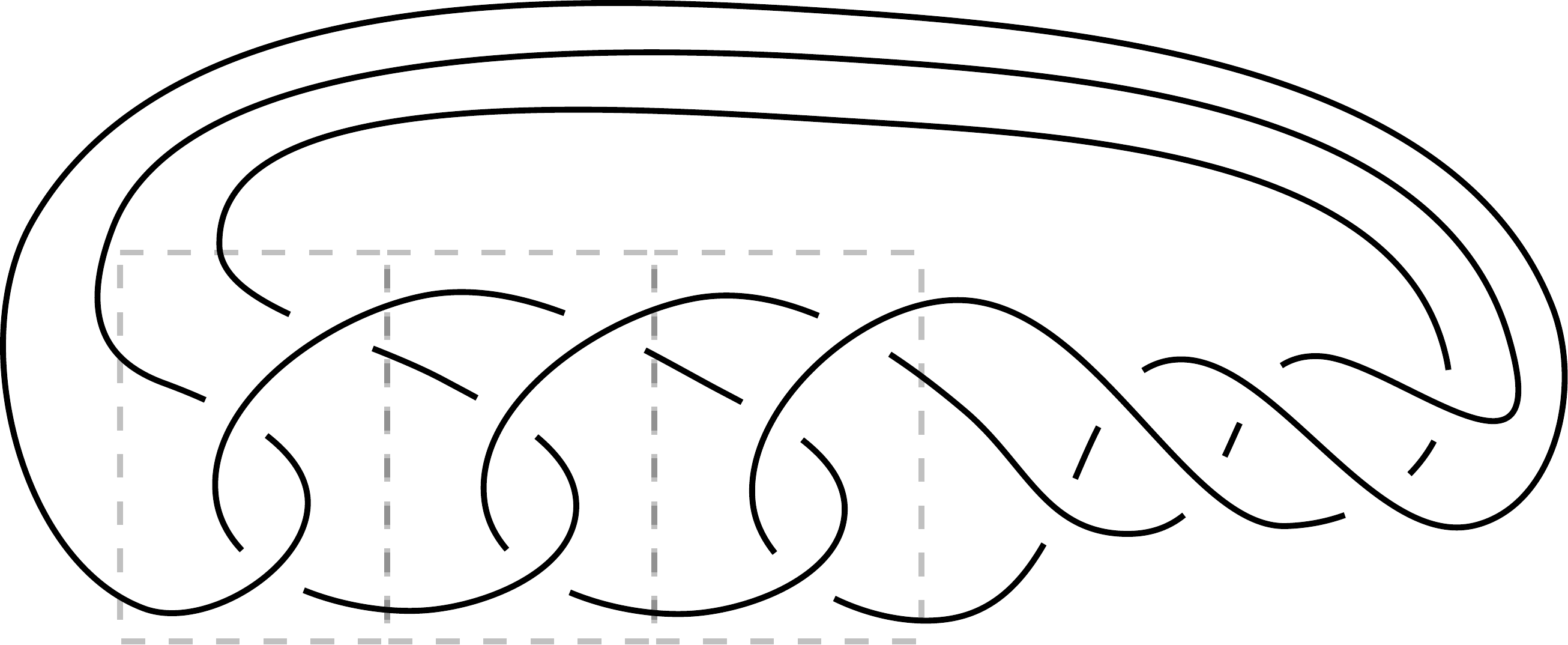}}
\caption{A freely $(3,1)$-periodic diagram for $T(2,7)$. This diagram consists of three identical tangles and one full twist.}
\label{fig:T(2,7)}
\end{figure}
The goal of this paper is to understand when a freely periodic knot bounds an equivariant orientable surface in $B^4$.

\begin{definition}
A freely $(p,q)$-periodic knot $K$ is an \emph{equivariant boundary} if there is a smooth order $p$ extension $\rho:B^4 \to B^4$ of the $\mathbb{Z}/p\mathbb{Z}$ symmetry on $S^3$ and an orientable surface $S$ properly smoothly embedded in $B^4$ with $\rho(S) = S$ and $\partial S  = K$.
\end{definition}

It is interesting to compare the case of periodic and strongly invertible knots, which always equivariantly bound an orientable surface in $S^3$ (and hence in $B^4$), see for example \cite[Proposition 1]{BoyleIssa}. On the other hand, freely periodic knots never bound equivariant orientable surfaces in $S^3$, since they would necessarily contain a fixed point by the Lefschetz fixed-point theorem, but the symmetry acts freely on $S^3$. 

In this paper we characterize which freely periodic knots equivariantly bound in $B^4$. 

\begin{restatable}{thm}{simple}
\label{thm:simple}
Let $K \subset S^3$ be a freely $(p,q)$-periodic knot and $\pi:S^3 \to L(p,q)$ be the quotient map. Then $K$ is an equivariant boundary if and only if $\pi(K)$ represents a simple homology class (see Definition \ref{def:simple}) in $H_1(L(p,q);\mathbb{Z})$.
\end{restatable}

Furthermore, we prove that a simple count of strands in a freely periodic diagram determines if the knot equivariantly bounds.

\begin{restatable}{thm}{simplediagram}
\label{thm:simplediagram}
Let $D$ be an oriented freely $(p,q)$-periodic diagram (see Definition \ref{def:diagram}) for a knot $K$. Let $m$ be the signed count of strands in $D$. Then $K$ represents a simple homology class in the quotient $L(p,q)$ if and only if $m \equiv \pm1$ or $m \equiv \pm q^{-1}$ (mod $p$).
\end{restatable}

In the case of torus knots, this reduces to a numerical condition on the torus knot parameters and the order of the symmetry. 

\begin{restatable}[Corollary of Theorem \ref{thm:simplediagram}]{cor}{torusknots}
\label{cor:torusknots}
A freely $(p,q)$-periodic symmetry of $T(r,s)$ equivariantly bounds if and only if $p$ is a divisor of $r+1,r-1,s+1,$ or $s-1$.
\end{restatable}

Throughout the paper, all statements are in the smooth category and all surfaces are orientable. 

\subsection{Acknowledgments}

We would like to thank Ahmad Issa for some helpful conversations and Robert Lipshitz for helpful comments on an earlier draft.

\section{Freely-periodic knots}

Our results are based on considering the homology class which a freely periodic knot represents in the quotient lens space $L(p,q)$. We begin by distinguishing some elements of $H_1(L(p,q))$.

\begin{definition} \label{def:simple}
A \emph{simple} homology class in $H_1(L(p,q);\mathbb{Z})$ is one which is represented by the core of a handlebody in a genus 1 Heegaard splitting.
\end{definition}
In fact there is a unique genus 1 Heegaard splitting for $L(p,q)$, see \cite{MR710104}, so that there are at most 4 simple homology classes in $H_1(L(p,q))$ coming from the two orientations on the cores of the two handlebodies.


The following proposition is a restatement of \cite[Lemma 5.2]{MR2884030}, noting that the lift in $S^3$ of a knot $\overline{K}$ in $L(p,q)$ is an unknot if and only if $\overline{K}$ is a rational unknot. 
\begin{proposition} \cite[Lemma 5.2]{MR2884030}\label{prop:unknotclasses}
Let $U$ be a freely $(p,q)$-periodic unknot. Then the quotient $\overline{U}$ represents a simple homology class in $H_1(L(p,q))$. Conversely, every simple homology class is represented by the quotient of an unknot. 
\end{proposition}

We now define our main object of study: equivariant cobordisms between freely periodic knots. As we will see in Lemma \ref{lemma:unknotcobordism}, studying equivariant cobordisms from $K$ to the unknot is equivalent to studying equivariant surfaces with boundary $K$. 

\begin{definition}
An \emph{equivariant cobordism} $(S^3 \times I, S, \overline{\rho})$ between a freely $(p,q)$-periodic knot $(K,\rho)$ and a freely $(p,q')$-periodic knot $(K',\rho')$ is a proper smooth embedding of a surface $S \to S^3 \times I$ such that $S \cap (S^3 \times \{0\}) = K$ and $S \cap (S^3 \times \{1\}) = K'$, along with a free smooth $\mathbb{Z}/p\mathbb{Z}$ action $\overline{\rho}$ on $S^3 \times I$ which restricts to the freely periodic symmetries $\rho$ on $S^3 \times \{0\}$ and $\rho'$ on $S^3 \times \{1\}$, and leaves $S$ invariant.
\end{definition}

To study these cobordisms, it will be convenient to first take the quotient by the free symmetry. 

\begin{lemma} \label{lemma:quotientcobordism}
Let $W = S^3 \times I$, let $\overline{\rho}$ be a finite order diffeomorphism acting freely on $W$, and let $\overline{\rho}|_{S^3 \times \{0\}} = \rho$ and $\overline{\rho}|_{S^3 \times \{1\}} = \rho'$; for example an equivariant cobordism between freely periodic knots. Then the quotient $\overline{W} = W/\overline{\rho}$ is a homology cobordism between lens spaces.
\end{lemma}

\begin{proof}
Since $S^3 \times I$ is simply connected, $\pi_1(\overline{W}) = \mathbb{Z}/p\mathbb{Z} = H_1(\overline{W})$, and the maps 
\[
H_1(S^3/\rho) \to H_1(\overline{W}) \mbox{ and } H_1(S^3/\rho') \to H_1(\overline{W})
\] 
induced by inclusion are isomorphisms. Since $W$ is connected, the same is true for $H_0$ so that $\overline{W}$ is a homology cobordism.
\end{proof}

We are interested in the maps induced on homology from these homology cobordisms. The following theorem, which follows from an analysis of the $d$-invariants of lens spaces, is not stated explicitly in \cite{DoigWehrli}, but follows immediately from their main theorem and \cite[Lemma 3]{DoigWehrli}.

\begin{thm}\cite{DoigWehrli} \label{thm:d-invariants}
Let $\overline{W}:L(p,q) \to L(p,q')$ be a homology cobordism. Then $L(p,q)$ is homeomorphic to $L(p,q')$ and $\overline{W}$ induces $\pm$\textup{Id} on $H_1(L(p,q);\mathbb{Z})$.
\end{thm}

\begin{corollary} \label{cor:qisq}
Let $K$ be a freely $(p,q)$-periodic knot and $K'$ be a freely $(p,q')$-periodic knot. If there is an equivariant cobordism between $K$ and $K'$, then $q' = \pm q^{\pm 1} \in \mathbb{Z}/p\mathbb{Z}$. In particular, $K'$ is a freely $(p,q)$-periodic knot. 
\end{corollary}
\begin{proof}
Apply Theorem \ref{thm:d-invariants} to the quotient of the equivariant cobordism, and use the homeomorphism classification of lens spaces \cite{MR116336}.
\end{proof}

We now relate equivariant cobordisms to equivariant surfaces for freely periodic knots.

\begin{lemma} \label{lemma:unknotcobordism}
A freely $(p,q)$-periodic knot $K$ is an equivariant boundary if and only if there is an equivariant cobordism between $K$ and a freely $(p,q)$-periodic unknot.
\end{lemma}
\begin{proof}
Let $S$ be an equivariant surface for $K$ with respect to an order $p$ diffeomorphism $\rho:B^4 \to B^4$. By classical Smith theory \cite{MR3199}, $\rho$ has a contractible fixed-point set $F$. Furthermore, since the fixed set of a self-diffeomorphism is a submanifold, $F$ is a single point. Similarly, the fixed-point set of $\rho|_S$ is also the single point $F$. Removing an equivariant neighborhood $N(F)$ from $B^4$ leaves an $S^3$ boundary component containing an unknot $U$. Hence we have $(S - N(F)) \subset S^3 \times I$ which is preserved by the free $\mathbb{Z}/p\mathbb{Z}$ action and $\partial (S - N(F)) = K \cup U$. That is an equivariant cobordism between $K$ and $U$. By Corollary \ref{cor:qisq}, $U$ comes with a freely $(p,q)$-periodic symmetry.

On the other hand, suppose that we have an equivariant cobordism between $K$ and $U$. By Proposition \ref{prop:unknotclasses}, $U$ can be taken to be the lift of a core of a handlebody in a genus 1 Heegaard decomposition of $L(p,q)$, and in particular the cone of $(S^3,U)$ is a smooth equivariant disk in $B^4$. Gluing this to the cobordism gives an equivariant surface for $K$. 
\end{proof}

We have all of the tools we need in order to classify which freely periodic knots bound equivariant surfaces in $B^4$. 

\simple*

\begin{proof}
Suppose $K \subset S^3$ is a freely $(p,q)$-periodic knot which bounds a surface $S \subset B^4$ which is invariant under an order $p$ diffeomorphism $\rho:B^4 \to B^4$ with $\rho |_{S^3}$ the free $\mathbb{Z}/p\mathbb{Z}$ action preserving $K$. Then by Lemma \ref{lemma:unknotcobordism} there is an equivariant cobordism from $K$ to a freely $(p,q)$-periodic unknot $U_{p,q}$. By Lemma \ref{lemma:quotientcobordism}, the quotient of this cobordism is a cobordism of the quotient knots $\overline{K}$ and $\overline{U}_{p,q}$ in a homology cobordism of lens spaces. In particular, there is a map $f:H_1(L(p,q)) \to H_1(L(p,q))$ with $[f(\overline{K})] = [\overline{U}_{p,q}]$ which is induced by a homology cobordism of lens spaces. By Theorem \ref{thm:d-invariants}, $f = \pm$Id. Hence $\overline{K}$ represents the same class as an unknot in $H_1(L(p,q))$. Then by Proposition \ref{prop:unknotclasses}, we have that $[\overline{K}]$ is simple.

On the other hand, suppose that $[\overline{K}]$ is simple. Then by Proposition \ref{prop:unknotclasses}, $\overline{K}$ is in the same homology class as the core $\overline{U}$ of a handlebody in the genus 1 Heegaard decomposition of $L(p,q)$, which lifts to a freely periodic unknot $U$. Hence there is a surface $\overline{S}$ in $L(p,q)$ with boundary $\overline{K} \cup \overline{U}$, and lifting $\overline{S}$ to $S^3$ gives us an equivariant surface $S$ with boundary $K \cup U$. We can then take the cone of $(S^3,U)$ with the free symmetry to obtain a smooth equivariant disk in $B^4$. Gluing this to $S$ gives us an equivariant surface for $K$. 
\end{proof}

We now turn to checking this condition in practice by considering freely periodic knot diagrams, the standard way in which we expect to present a freely periodic knot. 

\begin{definition} \label{def:diagram}
A freely $(p,q)$-periodic diagram is the closure of a tangle consisting of the concatenation of $p$ identical tangles and $q$ full twists. See Figures \ref{fig:T(2,7)}, \ref{fig:T(2,7)simple}, and \ref{fig:(5,1)} for examples.
\end{definition}

\begin{thm}\cite[Theorem 1]{MR1664974}
Every freely $(p,q)$-periodic knot has a freely $(p,q)$-periodic diagram. 
\end{thm}

From a freely periodic diagram, it is extremely simple to check if the knot is an equivariant boundary. 

\simplediagram*

\begin{proof}
It is clear that a tangle with one strand is homologous to the core $c$ of a genus 1 Heegaard splitting for $L(p,q)$. It follows immediately from the definition of a lens space that the other core is $q^{-1} \cdot c \in H_1(L(p,q))$ so that a diagram with $q^{-1}$ strands is homologous to the other core. 
\end{proof}

\begin{remark}
Theorem \ref{thm:simplediagram} depends on our convention for freely $(p,q)$-periodic diagrams. We could instead switch the role of the two cores in our diagrams so that a freely $(p,q)$-periodic diagram would have $q^{-1}$ (mod $p$) full twists. In this case the simple homology classes would be represented by diagrams with $\pm1$ or $\pm q$ strands. 
\end{remark}

\begin{remark}
If we instead consider surfaces which need not be orientable, then a version of Theorem \ref{thm:simple} using homology with $\mathbb{Z}/2\mathbb{Z}$ coefficients implies that every freely periodic knot bounds a smooth equivariant surface in $B^4$.
\end{remark}

We are also interested in the special case of torus knots. We first describe their free symmetries. 

\begin{thm} \label{thm:torusknotsymmetries}
There is a freely $(p,q)$-periodic symmetry of $T(r,s)$ if and only if gcd$(p,rs) = 1)$, and $\pm q^{\pm1} \equiv sr^{-1}$ (mod $p$). With respect to this symmetry, there is a freely $(p,q)$-periodic diagram on $r$ strands if and only if $\pm q \equiv sr^{-1}$ (mod $p$). Finally, if $\rho$ and $\rho'$ are two free periods of the same order, then the subgroups of the orientation-preserving diffeomorphism group of the exterior of $T(r,s)$ generated by $\rho$ and $\rho'$ are conjugate.
\end{thm}

\begin{proof}

By \cite[Theorem 2.2]{MR856847} and \cite{MR3069424} (see also \cite[Exercise 10.6.4]{MR1417494}), the freely periodic symmetry preserves a Seifert fibered structure on the exterior of $T(r,s)$, which is a circle bundle over $D(r,s)$, the disk with cone points of orders $r$ and $s$. For $p>2$, this orbifold has no order $p$ symmetries. For the involution reflecting across an arc containing the two cone points, we must also reflect the fibers to get an orientation-preserving symmetry of the exterior of $T(r,s)$, but this has global fixed points. Thus our free symmetry acts as identity on the base orbifold, and therefore as a rotation on the fibers. In particular, there is a unique symmetry for each $p$. 


When $p$ has a common factor with $r$ or $s$, then the order gcd$(p,rs)$ subgroup has a fixed circle in $S^3$, and so there is no free symmetry of order $p$. On the other hand, the symmetry is free when gcd$(p,rs) = 1$. We will describe these free symmetries explicitly. Consider the standard tangle diagram for $T(r,s)$ with $r$ strands. Then there are $s$ total twists, which can be grouped as $\pm q$ full twists and $p$ concatenated tangles of $n$ twists each (see Figure \ref{fig:T(2,7)simple} for an example). Hence $s = \pm q \cdot r + n \cdot p$ from which we deduce that $\pm q \equiv sr^{-1}$ (mod $p$). However, a free $(p,q)$-symmetry can also be regarded as a free $(p,q^{-1})$-symmetry (c.f. the homeomorphism classification of lens spaces). In particular the diagram with $s$ strands has $\pm q \equiv rs^{-1}$ (mod $p$).
\end{proof}

Theorem \ref{thm:simplediagram} gives us the following entirely numerical corollary when considering torus knots. 

\torusknots*

\begin{proof}
Consider the freely $(p,q)$-periodic diagram for $T(r,s)$ with $r$ strands which consists of $s$ total twists: $\pm q$ full twists and $p$ tangles of $n$ twists each. By Theorem \ref{thm:simple} and Theorem \ref{thm:simplediagram}, $T(r,s)$ is an equivariant boundary if and only if $r \equiv \pm 1$ or $r \equiv \pm q^{-1}$ (mod $p$). By Theorem \ref{thm:torusknotsymmetries}, $\pm q \equiv sr^{-1}$ (mod $p$) so that $p$ is a divisor of $r-1,r+1,s-1,$ or $s+1$.
\end{proof}

\begin{remark}
Note that a freely $(2,q)$ or $(3,q)$-periodic knot is always an equivariant boundary by Theorem \ref{thm:simplediagram}.
\end{remark}

Our final result is the observation that the genus of an equivariant surface for a freely periodic knot must be a multiple of $p$.

\begin{thm} \label{thm:genus}
Let $S$ be an equivariant surface for a freely $(p,q)$-periodic knot. Then the genus of $S$ is a multiple of $p$. 
\end{thm}
\begin{proof}
Let $n$ be the genus of $S$ so that $\chi(S) = 1-2n$. Since $S$ has exactly one fixed point $F$ (see the proof of Lemma \ref{lemma:unknotcobordism}), $\chi(S - F) = -2n$. Quotienting $S-F$ by the free $\mathbb{Z}/p\mathbb{Z}$ action leaves a surface $\overline{S}$ with $\chi(\overline{S}) = -2n/p$. But since $\overline{S}$ has exactly two boundary components, it must have an even Euler characteristic and so $p|n$.
\end{proof}

\section{Examples and questions}

We conclude with some examples and questions.

\begin{example}
Consider the torus knot $T(2,7)$. It has a unique free period of order 3 which is shown in Figure \ref{fig:T(2,7)simple}. By Corollary \ref{cor:torusknots}, $T(2,7)$ bounds an invariant surface in $B^4$ with respect to its free $(3,2)$-period. This surface can be seen directly from Figure \ref{fig:T(2,7)}, where changing the sign of the crossing in the bottom of each of the 3 tangle boxes gives an equivariant genus 3 cobordism to the unknot, which then bounds an equivariant disk in $B^4$. Note that by Theorem \ref{thm:genus}, this is the minimum possible genus for such a surface, since $T(2,7)$ is not slice.

\begin{figure} 
\scalebox{.5}{\includegraphics{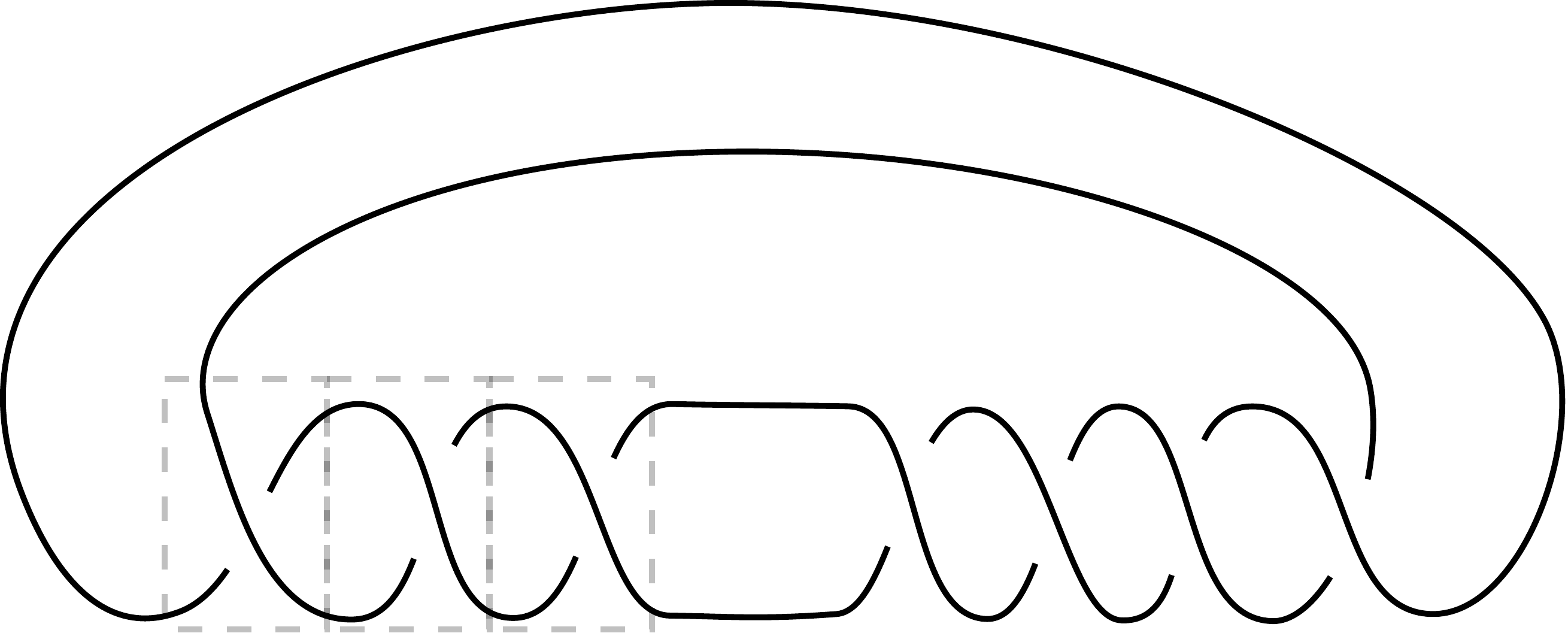}}
\caption{A freely $(3,2)$-periodic diagram for $T(2,7)$.}
\label{fig:T(2,7)simple}
\end{figure}

\end{example}

\begin{figure}
\scalebox{.5}{\includegraphics{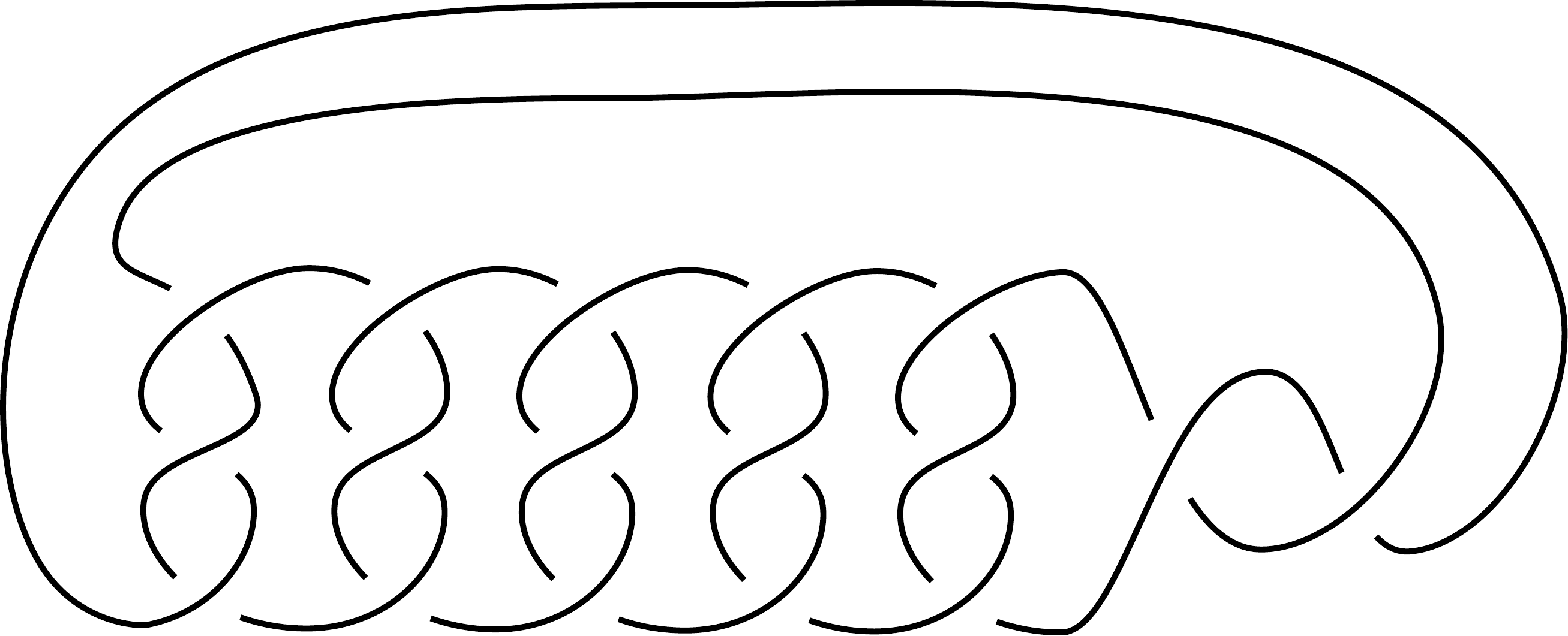}}
\caption{A freely $(5,1)$-periodic knot.}
\label{fig:(5,1)}
\end{figure}

\begin{example}
Consider the torus knot $T(2,3)$. Since there are no primes which divide $1,3,2,$ or $4$ and are relatively prime to $2$ and $3$, Corollary \ref{cor:torusknots} implies that $T(2,3)$ does not equivariantly bound with respect to any of its infinitely many free symmetries.
\end{example}

\begin{example}
Consider the freely $(5,1)$-periodic knot $K$ shown in Figure \ref{fig:(5,1)}. It has a signed count of 2 strands, so by Theorem \ref{thm:simple} and Theorem \ref{thm:simplediagram}, $K$ does not bound an equivariant surface in $B^4$. 

\end{example}

In principle one could use Theorem \ref{thm:simple} and Theorem \ref{thm:simplediagram} to show that some freely-periodic slice knots do not bound equivariant surfaces (let alone disks). However, every freely-periodic slice knot we know of represents a simple homology class in the quotient lens space.

\begin{question}
What tools can give a lower bound on the genus of an equivariant surface for a freely-periodic knot? Does there exist a freely-periodic slice knot which does not bound any equivariant disk?
\end{question}

Our main result requires methods relying on smooth topology. However, we do not know of a freely $(p,q)$-periodic knot which does not represent a simple homology class in $L(p,q)$, but which is an equivariant boundary in the topologically locally flat category.

\begin{question}
Does Theorem \ref{thm:simple} hold in the topological category?
\end{question}
\pagebreak

\bibliography{bibliography}
\bibliographystyle{alpha}
\end{document}